\documentclass[12pt]{amsart}
\usepackage[utf8]{inputenc}
\usepackage{amsfonts}
\usepackage{amsmath}
\usepackage{amssymb}
\usepackage{amsthm}
\usepackage{bbm}
\usepackage{xcolor}
\usepackage[margin=1.3in]{geometry}
\usepackage{graphicx}
\usepackage{fancyvrb}
\usepackage[11pt]{moresize}
\usepackage{tikz}
\tikzstyle{vertex}=[circle,draw=black,fill=black,inner sep=0,minimum size=3pt,text=white,font=\footnotesize]
\usepackage{nccmath}

\newtheorem{theorem}{Theorem}
\newtheorem*{conjecture*}{Conjecture}
\newtheorem{proposition}{Proposition}[section]
\newtheorem{lemma}[proposition]{Lemma}

\newtheorem{definition}[proposition]{Definition}

\theoremstyle{definition}
\newtheorem{remark}[proposition]{Remark}

\newtheorem{manualtheoreminner}{Proposition}
\newenvironment{manualtheorem}[1]{%
  \renewcommand\themanualtheoreminner{#1}%
  \manualtheoreminner
}{\endmanualtheoreminner}

\newcommand{\vs}{\vspace{3mm}}
\newcommand{\vsss}{\vspace{6mm}}
\newcommand{\un}{\underline}
\newcommand{\hs}{\hspace{1mm}}

\newcommand{\R}{\mathbb{R}}
\newcommand{\Z}{\mathbb{Z}}

\newcommand{\N}{\mathbb{N}}
\newcommand{\C}{\mathbb{C}}
\newcommand{\T}{\mathbb{T}}

\newcommand{\mf}{\mathfrak}

\newcommand{\ep}{\epsilon}
\newcommand{\lam}{\lambda}
\newcommand{\sub}{\subseteq}
\newcommand{\al}{\alpha}

\newcommand{\modd}[1]{\text{ mod } #1}

\newcommand{\lcm}{\text{lcm}}
\newcommand{\ol}{\overline}
\newcommand{\wt}{\widetilde}
\newcommand{\wh}{\widehat}

\title{On Sumsets Containing a Perfect Square}
\author{Zachary Chase}
\thanks{The author is partially supported by Ben Green's Simons Investigator Grant 376201 and gratefully acknowledges the support of the Simons Foundation.}
\address{Mathematical Institute, Andrew Wiles Building, Radcliffe Observatory Quarter, Woodstock Road, Oxford OX2 6GG, UK}
\email{zachary.chase@maths.ox.ac.uk}
\date{January 11, 2022}

\begin{document}

\begin{abstract}
We show $A+B$ contains a perfect square if $A,B \sub \{1,\dots,N\}$ have $|A|,|B| \ge (\frac{3}{8}+\epsilon)N$. The constant $\frac{3}{8}$ is optimal. 
\end{abstract}

\maketitle

\section{Introduction}

Let $A,B$ be subsets of the first $N$ positive integers. What are the maximum possible sizes of $A$ and $B$ if $A+B$ does not contain a perfect square? 

\vs

Let us first discuss the history of the related question of the largest size of a subset $A \sub \{1,\dots,N\}$ with $A+A$ not containing a perfect square, originally raised by Erd\H{o}s and Silverman \cite[p.~87,~107]{erdosandsilverman}. Erd\H{o}s initially conjectured that the answer is roughly $\frac{1}{3}N$, coming from $$A := \{n \le N : n \equiv 1 \modd 3\}.$$ However, Massias \cite{massias} noted that $$A := \{n \le N : n \modd 32 \in \{1,5,9,13,14,17,21,25,26,29,30\}\}$$ gives the larger size of roughly $\frac{11}{32}N$. The two mentioned sets $A$ indeed have the property that $A+A$ does not contain a perfect square, since the sumset of $\{1\} \sub \Z/3\Z$ with itself does not contain a quadratic residue (in $\Z/3\Z$), and the sumset of $\{1,5,9,13,14,17,21,25,26,29,30\} \sub \Z/32\Z$ with itself avoids quadratic residues. 

\vs

Given that these two examples come from ``lifting up" a set $A \sub \Z/q\Z$ for some $q \in \N$, and that any perfect square must be a quadratic residue mod $q$, it is natural to first solve the ``modular" version of the problem: for given $q \in \N$, what is the largest size of a set $A \sub \Z/q\Z$ such that $A+A$ does not contain a quadratic residue? 

\vs

In 1982, Lagarias, Odlyzko, and Shearer \cite{los1} showed the answer is $\frac{11}{32}q$ (which is tight if $32 \mid q$). In 1983, they released a companion paper \cite{los2} proving that if $A \sub [N]$ has $|A| \ge 0.475 N$ then $A+A$ contains a perfect square. Finally, in 2001, Khalfalah, Lodha, and Szemerédi \cite{kls} resolved the Erd\H{o}s-Silverman problem, by showing that for all $\epsilon > 0$, if $N$ is sufficiently large, then any $A \sub [N]$ with $A+A$ avoiding perfect squares must have $|A| \le (\frac{11}{32}+\epsilon)N$. 

\vs

In this paper, we solve the aformentioned ``bipartite" version of the Erd\H{o}s-Silverman question. Our result is asymptotically optimal. 

\vspace{1.5mm}

\begin{theorem}\label{3/8+epsilon}
For any $\ep > 0$, if $N$ is sufficiently large and $A,B \sub [N]$ have $|A|,|B| \ge (\frac{3}{8}+\ep)N$, then $A+B$ contains a perfect square.
\end{theorem}

\vspace{1mm}

An example achieving roughly $\frac{3}{8}N$ is $$A := \{n \le N : n \modd 8 \in \{0,1,5\}\}$$ $$B := \{n \le N : n \modd 8 \in \{2,5,6\}\},$$ which works since the $\Z/8\Z$-sumset $\{0,1,5\}+\{2,5,6\}$ avoids quadratic residues.

\vs

We prove Theorem \ref{3/8+epsilon} by first resolving the associated ``modular" version of the problem. While the methods of \cite{los1}, solving the modular problem for $A+A$, are highly graph-theoretic, our methods use Fourier analysis to reduce (in one direction) to solving some optimization problem in $48$ variables. Interestingly, the paper \cite{los1} also involved solving some optimization problems, specifically various integer programs. It is plausible our methods could solve the modular $A+A$ problem, though the number of variables in the obtained optimization problem would be significantly too large. 

\vs

We then obtain the result in the integers by basic Fourier-analytic arguments. While \cite{kls}, solving the $A+A$ problem in the integers, introduced a novel ``shifting method" and a low-level strong arithmetic regularity lemma with tower-type bounds, our Fourier arguments amount to a rather basic arithmetic regularity lemma with only singly exponential bounds. In rough terms, we approximate the characteristic function of $A \sub [N]$ (and of $B$) by its best modulo $Q$ weight function approximation on $\eta^{-1}$ intervals each of length $\eta N$, where $\eta^{-1}$ and $\log Q$ are polynomials of $\epsilon^{-1}$. Counting the number of perfect squares ``in" the convolution of these weight functions essentially reduces to the modular problem. For details, see Section \ref{integersection}.

\vspace{1.5mm}

\section{Notation}

We use the standard $[N] := \{1,\dots,N\}$ and $e(\theta) := e^{2\pi i \theta}$. Let $\frac{1}{\N} := \{\frac{1}{n} : n \ge 1\}$. Let $\T = \R/\Z$. For $f: [N] \to \C$, define $\wh{f} : \T \to \C$ by $$\wh{f}(\theta) := \sum_{n \le N} f(n)e(-n\theta).$$ For $f : \Z/q\Z \to \C$, define $\wh{f} : \Z/q\Z \to \C$ by $$\wh{f}(r) := \frac{1}{q}\sum_{x \in \Z/q\Z} f(x)e\left(-\frac{rx}{q}\right).$$ Define the weighted indicator function of the quadratic residues $f_q : \Z/q\Z \to \R$ by $$f_q(t) := |\{x \in \Z/q\Z : x^2 = t\}|.$$ For functions $f,g : \Z/q\Z \to \C$, define the convolution of $f,g$ as $$(f*g)(x) := \frac{1}{q}\sum_{a \in \Z/q\Z} f(a)g(x-a),$$ while for finitely supported functions $f,g : \Z \to \C$, we define the convolution as $$(f*g)(x) := \sum_{n \in \Z} f(n)g(x-n).$$

\vspace{1.5mm}

\section{The Modular Problem}\label{modularsection}

In this section, we prove the following, a (doubly) weighted, quantitative version of the statement that $A+B$ contains a quadratic residue if $A,B \sub \Z/q\Z$ have $|A|,|B| > \frac{3}{8}q$.

\begin{theorem}\label{modular}
For any $\epsilon > 0$ there is some $c(\epsilon) > 0$ so that for any $q \ge 1$, if $w_A,w_B : \Z/q\Z \to [0,1]$ have $\sum_{t \in \Z/q\Z} w_A(t), \sum_{t \in \Z/q\Z} w_B(t) \ge (\frac{3}{8}+\epsilon)q$, then $$\sum_{t \in \Z/q\Z} (w_A*w_B)(t)f_q(t) \ge c(\epsilon)q.$$ In fact, one can take $c(\ep) = \frac{1}{\sqrt{5}}\ep$. 
\end{theorem}

Our approach is Fourier-analytic. We start by noting the Fourier representation of this weighted count of quadratic residues ``in" the convolution of $w_A$ and $w_B$. 

\begin{lemma}\label{fourierrepresentation}
For any $w_A,w_B : \Z/q\Z \to \R$, we have $$\frac{1}{q}\sum_{t \in \Z/q\Z} (w_A*w_B)(t)f_q(t) = \sum_{m \in \Z/q\Z} \wh{w_A}(m)\wh{w_B}(m)\wh{f_q}(-m).$$
\end{lemma}

\begin{proof}
The right hand side is, by definition, equal to $$\sum_{m \in \Z/q\Z} \frac{1}{q^3}\sum_{x,y,z \in \Z/q\Z} w_A(x)w_B(y)f_q(z)e\left(\frac{m(z-x-y)}{q}\right).$$ Interchanging summations and using the orthogonality condition $$\vspace{-2mm} \sum_{m \in \Z/q\Z} e\left(\frac{mr}{q}\right) = \begin{cases} q & \text{ if } r \equiv 0 \modd q \\ 0 & \text{ if } r \not \equiv 0 \modd q \end{cases}$$ finishes the proof.
\end{proof}

\vspace{2mm}

\begin{remark}\label{mod8partremark}
Let us take a moment to motivate the arguments to come. Suppose for now $q$ is divisible by $8$. We (a posteriori) expect $\sum_{t \in \Z/q\Z} (w_A*w_B)(t)f_q(t)$ to be minimized by weights $w_A,w_B$ that are ``lift-ups" of weights $\ol{w}_A,\ol{w}_B : \Z/8\Z \to [0,1]$ in the sense\footnote{Note ``\text{mod} $8$" makes sense since $8 \mid q$.} $w_A(t) = \ol{w}_A(t \modd 8)$ and $w_B(t) = \ol{w}_B(t \modd 8)$. If $w_A$ and $w_B$ were indeed of this form, then, as one may easily check, we would have $\wh{w}_A(m),\wh{w}_B(m) = 0$ for each $m \in \Z/q\Z$ with $\frac{q}{\gcd(q,m)} \nmid 8$. Therefore, in our setting (in which $w_A,w_B$ might not be exactly of that form), it's natural to separate\footnote{Note that $\frac{q}{\gcd(q,m)} \nmid 8$ is equivalent to $q \nmid 8m$.}, $$\sum_{m \in \Z/q\Z} \wh{w_A}(m)\wh{w_B}(m)\wh{f_q}(-m) = \sum_{q \mid 8m} \wh{w_A}(m)\wh{w_B}(m)\wh{f_q}(-m) + \sum_{q \nmid 8m} \wh{w_A}(m)\wh{w_B}(m)\wh{f_q}(-m).$$ The latter term we shall upper-bound in magnitude, using that $\wh{f}_q(-m)$ is small for all $m$ with $q \nmid 8m$ (this follows from quadratic Gauss sum bounds). And the first term actually turns out to be just the weighted count of mod $8$ quadratic residues in the weighted sumset of the mod $8$ projections of the weight functions $w_A,w_B$. 
\end{remark}

\vspace{1.5mm}

For technical reasons, we work mod $24$ instead of mod $8$.    

\begin{lemma}\label{mod8part}
Let $q \in \N$ be a multiple of $24$. Let $w_A,w_B : \Z/q\Z \to [0,1]$ be two (weight) functions, and let $a,b : \Z/24\Z \to [0,1]$ be the mod $24$-projections of $w_A,w_B$: $$a(k) := \frac{1}{q/24}\sum_{\substack{x \in \Z/q\Z \\ x \equiv k {\normalfont \hspace{0.8mm}  \text{mod} \hspace{0.8mm}} 24}} w_A(x)$$ $$b(k) := \frac{1}{q/24}\sum_{\substack{x \in \Z/q\Z \\ x \equiv k {\normalfont \hspace{0.8mm}  \text{mod} \hspace{0.8mm}} 24}} w_B(x).$$ Then one has $$\sum_{\substack{m \in \Z/q\Z \\ q \mid 24m}} \wh{w_A}(m)\wh{w_B}(m)\wh{f_q}(-m) = \frac{1}{24}\sum_{t \in \Z/24\Z} (a*b)(t)f_{24}(t).$$ 
\end{lemma}

\begin{proof}
Noting $q \mid 24m$ if and only if $m = \frac{rq}{24}$, we may write the LHS as $$\sum_{r=0}^{23} \frac{1}{q^3}\sum_{x,y,z \in \Z/q\Z} w_A(x)w_B(y)f_q(z)e\left(\frac{rq}{24}\hs\frac{z-x-y}{q}\right),$$ which by orthogonality (mod $24$) is equal to $$\frac{24}{q^3}\sum_{\substack{x,y,z \in \Z/q\Z \\ x+y \equiv z \modd 24}} w_A(x)w_B(y)f_q(z).$$ Splitting into cases mod $24$, we may write the above as \begin{equation}\label{pota} \frac{24}{q^3}\sum_{i,j \in \Z/24\Z} \left(\sum_{\substack{x \in \Z/q\Z \\ x \equiv i \modd 24}} w_A(x)\right)\left(\sum_{\substack{y \in \Z/q\Z \\ y \equiv j \modd 24}} w_B(y)\right)\left(\sum_{\substack{z \in \Z/q\Z \\ z \equiv i+j \modd 24}} f_q(z)\right).\end{equation} Noting $$\sum_{\substack{z \in \Z/q\Z \\ z \equiv i+j \modd 24}} f_q(z) = \sum_{\substack{z \in \Z/q\Z \\ z \equiv i+j \modd 24}} \sum_{v \in \Z/q\Z} 1_{v^2 \equiv z \modd q} = \sum_{v \in \Z/q\Z} 1_{v^2 \equiv i+j \modd 24} = \frac{q}{24}f_{24}(i+j),$$ and using the definitions of $a,b$, we may write \eqref{pota} as $$\frac{1}{24^2}\sum_{i,j \in \Z/24\Z} a(i)b(j)f_{24}(i+j) = \frac{1}{24}\sum_{t \in \Z/24\Z} (a*b)(t)f_{24}(t),$$ as desired.
\end{proof}

\vspace{1.5mm}

We now go on to handle the other Fourier term, $\sum_{q \hspace{0.5mm} \nmid \hspace{0.5mm} 24m} \wh{w}_A(m)\wh{w}_B(m)\wh{f}_q(-m)$. 

\vspace{1.5mm}

\begin{lemma}\label{gausssumbound}
Let $q \in \N$ be a multiple of $24$. Then for any $m \in \Z$ with $q \hspace{0.5mm} \nmid \hspace{0.5mm} 24m$, one has $$\left|\wh{f}_q(-m)\right| \le \frac{1}{\sqrt{5}}.$$ 
\end{lemma}

\begin{proof}
By definition, $$\wh{f_q}(-m) = \frac{1}{q}\sum_{t \in \Z/q\Z} \left(\sum_{x \in \Z/q\Z} 1_{x^2 \equiv t}\right) e(\frac{mt}{q}) = \frac{1}{q}\sum_{x \in \Z/q\Z} e\left(\frac{mx^2}{q}\right) = \frac{1}{q/g}\sum_{x \in \Z/\frac{q}{g}\Z} e\left(\frac{\frac{m}{g}x^2}{q/g}\right),$$ where $g := \gcd(m,q)$. Thus, by standard quadratic Gauss sum estimates (e.g., \cite{irelandrosen}), $$\left|\wh{f_q}(-m)\right| \le \begin{cases} \sqrt{\frac{1}{q/g}} & \text{if } q/g \in \{1,3\} \modd 4 \\ \sqrt{\frac{2}{q/g}} & \text{if } q/g \equiv 0 \modd 4 \\ 0 & \text{if } q/g \equiv 2 \modd 4. \end{cases}$$ Now, $q \hspace{0.5mm} \nmid \hspace{0.5mm} 24m$ implies $\frac{q}{g} \nmid 24$. This implies, firstly, that $\frac{q}{g} \ge 5$, giving $\sqrt{\frac{1}{q/g}} \le \frac{1}{\sqrt{5}}$, and, secondly, that if $\frac{q}{g} \equiv 0 \modd 4$, then $\frac{q}{g} \ge 16$, giving $\sqrt{\frac{2}{q/g}} \le \frac{1}{\sqrt{8}} \le \frac{1}{\sqrt{5}}$.
\end{proof}

\vspace{1.5mm}

\begin{lemma}\label{linfinityandl2}
Let $q \in \N$ be a multiple of $24$. Let $w_A,w_B : \Z/q\Z \to [0,1]$ be two (weight) functions, and let $a,b : \Z/24\Z \to [0,1]$ be the projections of $w_A,w_B$ mod $24$ as in Lemma \ref{mod8part}. Then, $$\left|\sum_{\substack{m \in \Z/q\Z \\ q \hspace{0.5mm} \nmid \hspace{0.5mm} 24m}} \wh{w_A}(m)\wh{w_B}(m)\wh{f_q}(-m)\right| \le \frac{1}{24\sqrt{5}}\sqrt{\sum_{k \in \Z/24\Z} \left(a(k)-a(k)^2\right)}\sqrt{\sum_{k \in \Z/24\Z} \left(b(k)-b(k)^2\right)}.$$ 
\end{lemma}

\begin{proof}
By Lemma \ref{gausssumbound} and Cauchy-Schwarz, we have \begin{align*}\left|\sum_{\substack{m \in \Z/q\Z \\ q \hspace{0.5mm} \nmid \hspace{0.5mm} 24m}} \wh{w_A}(m)\wh{w_B}(m)\wh{f_q}(-m)\right| &\le \left(\sup_{\substack{m \in \Z/q\Z \\ q \hspace{0.5mm} \nmid \hspace{0.5mm} 24m}} |\wh{f_q}(-m)|\right)\left(\sum_{\substack{m \in \Z/q\Z \\ q \hspace{0.5mm} \nmid \hspace{0.5mm} 24m}} |\wh{w_A}(m)|\hs|\wh{w_B}(m)|\right) \\ &\le \frac{1}{\sqrt{5}}\sqrt{\sum_{\substack{m \in \Z/q\Z \\ q \hspace{0.5mm} \nmid \hspace{0.5mm} 24m}} |\wh{w_A}(m)|^2}\sqrt{\sum_{\substack{m \in \Z/q\Z \\ q \hspace{0.5mm} \nmid \hspace{0.5mm} 24m}} |\wh{w_B}(m)|^2}.\end{align*} The following two (in)equalities (and their analogues for $B$) finish the proof: $$\hspace{-4mm} \sum_{\substack{m \in \Z/q\Z \\ q \mid 24m}} |\wh{w_A}(m)|^2 = \sum_{r=0}^{23} \frac{1}{q^2}\sum_{x,y \in \Z/q\Z} w_A(x)w_A(y) e\left(\frac{r(x-y)}{24}\right)$$ $$ \hspace{42mm} = \frac{24}{q^2} \sum_{i \in \Z/24\Z} \left(\sum_{\substack{x \in \Z/q\Z \\ x \equiv i \modd 24}} w_A(x)\right)^2 = \frac{1}{24}\sum_{k \in \Z/24\Z} a(k)^2.$$ $$\hspace{3.5mm} \sum_{m \in \Z/q\Z} |\wh{w_A}(m)|^2 = \sum_{m \in \Z/q\Z} \frac{1}{q^2} \sum_{x,y \in \Z/q\Z} w_A(x)w_A(y)e\left(\frac{m(x-y)}{q}\right)$$ $$\hspace{45.5mm} = \frac{1}{q}\sum_{x \in \Z/q\Z} w_A(x)^2 \le \frac{1}{q}\sum_{x \in \Z/q\Z} w_A(x) = \frac{1}{24}\sum_{k \in \Z/24\Z} a(k).$$
\end{proof}

Combining Lemmas \ref{fourierrepresentation}, \ref{mod8part}, and \ref{linfinityandl2} (and multiplying through by $24$) yields \begin{equation}\label{key} \hspace{-45mm} \frac{24}{q}\sum_{t \in \Z/q\Z} (w_A*w_B)(t)f_q(t) \ge \sum_{t \in \Z/24\Z} (a*b)(t)f_{24}(t)\end{equation} $$\hspace{60mm} -\frac{1}{\sqrt{5}}\sqrt{\sum_{k \in \Z/24\Z} \left(a(k)-a(k)^2\right)}\sqrt{\sum_{k \in \Z/24\Z} \left(b(k)-b(k)^2\right)}.$$ Note that $a(k) \in [0,1]$ for each $k$ and that $$\sum_{k \in \Z/24\Z} a(k) = 24\cdot\frac{1}{q}\sum_{x \in \Z/q\Z} w_A(x),$$ implying $\sum_{k \in \Z/24\Z} a(k) \ge 9+24\ep$ if $\sum_{x \in \Z/q\Z} w_A(x) \ge (\frac{3}{8}+\ep)q$. We prove the following proposition in Section \ref{optimizationsection}. We assume it to be true for the rest of this section. In it, we use the notation $a(i) := a_i, b(i) := b_i$. We emphasize that it is ``merely" a (quadratic) optimization problem in $48$ variables.

\begin{proposition}\label{optimization0}
For any $\ep > 0$, there is some $c'(\ep) > 0$ so that the following holds. For all $a_0,\dots,a_{23},b_0,\dots,b_{23} \in [0,1]$ with $\sum_{i=0}^{23} a_i \ge 9+\ep, \sum_{i=0}^{23} b_i \ge 9+\ep$, one has $$\sum_{t \in \Z/24\Z} (a*b)(t)f_{24}(t) \ge c'(\ep)+\frac{1}{\sqrt{5}}\sqrt{\sum_i a_i-\sum_i a_i^2}\sqrt{\sum_i b_i-\sum_i b_i^2}.$$ In fact, one can take $c'(\ep) = \frac{1}{\sqrt{5}}\ep$.  
\end{proposition}

\vspace{1.5mm}

\begin{proof}[Proof of Theorem \ref{modular}]
If $24 \mid q$, then Theorem \ref{modular} follows immediately from \eqref{key} and Proposition \ref{optimization0} (with $c(\ep) = c'(24\ep)/24$). Otherwise, we use a simple ``lift-up" argument to reduce to the case $q \mid 24$. Define $\wt{w}_A,\wt{w}_B : \Z/24q\Z \to [0,1]$ by $\wt{w}_A(x) := \frac{1}{24}\sum_{\substack{y \in \Z/24\Z \\ y \equiv x \modd q}} w_A(y), \wt{w}_B(x) := \frac{1}{24}\sum_{\substack{y \in \Z/24\Z \\ y \equiv x \modd q}} w_B(y)$. Then $$\frac{1}{q}\sum_{t \in \Z/q\Z} (w_A*w_B)(t)f_q(t) = \frac{1}{24q}\sum_{t \in \Z/24q\Z} (\wt{w}_A*\wt{w}_B)(t)f_{24q}(t)$$ and $$\frac{1}{24q}\sum_{x \in \Z/24q\Z} \wt{w}_A(x) = \frac{1}{q}\sum_{x \in \Z/q\Z} w_A(x)$$ $$\frac{1}{24q}\sum_{x \in \Z/24q\Z} \wt{w}_B(x) = \frac{1}{q}\sum_{x \in \Z/q\Z} w_B(x).$$
\end{proof}

\vspace{1.5mm}

\section{Converting to Integers}\label{integersection}

In this section, we ``boost" the solution to the modular problem (Theorem \ref{modular}) to the integers to establish our main theorem (Theorem \ref{3/8+epsilon}). For subsets $A,B \sub [N]$ with $|A|,|B| \ge (\frac{3}{8}+\ep)N$ we shall, as in the modular problem, look at the number of squares in the weighted sumset of $A$ and $B$: $$\sum_{n \ge 1} (1_A*1_B)(n)1_S(n),$$ where $S \sub \N$ is the set of perfect squares, $S := \{m^2 : m \in \N\}$. Our approach is inspired by the arithmetic regularity lemma (see, e.g., \cite{eberhard, greentao}), though a much lower-tech version suffices for our purposes; the dependence on the relevant parameters will be singly-exponential rather than tower-type. 

\vs

\begin{definition}
Fix (parameters) $Q \in \N$ and $\eta \in \frac{1}{\N}$. For $k \in \{0,1,\dots,\eta^{-1}-1\}$, let \vspace{1.3mm} $$I_{\eta,k} = \Big(k\eta N, (k+1)\eta N\Big]\cap \N. \vspace{1mm}$$ For $N \in \N$ (large) and $A \sub [N]$, define\footnote{Extend (the domain of) $w^A_{Q;\eta,k}$ to $\N$ by setting $w^A_{Q;\eta,k} = 0$ outside $I_{\eta,k}$.} the function $w^A_{Q;\eta,k} : I_{\eta,k} \to [0,1]$ by $$w^A_{Q;\eta,k}(n) := \frac{\#\{m \in I_{\eta,k} : m \in A \text{ and } \hspace{0.5mm} m \equiv n {\normalfont \hspace{0.8mm}  \text{mod} \hspace{1.2mm}} Q\}}{\#\{m \in I_{\eta,k} : m \equiv n {\normalfont \hspace{0.8mm}  \text{mod} \hspace{1.2mm}} Q\}}.$$ Finally, define the function $w^A_{Q;\eta} : \N \to [0,1]$ by $$w_{Q;\eta}^A := \sum_{k=0}^{\eta^{-1}-1} w^A_{Q;\eta,k}1_{I_{\eta,k}}.$$
\end{definition}

\begin{remark}\label{approximantsintuition}
One should think of the function $w^A_{Q;\eta,k}$ as the best mod $Q$ approximation to $A$, or as a ``smoothed out" version of $A$ modulo $Q$, on $I_{\eta,k}$. Indeed, for $n \in I_{\eta,k}$, the function $w^A_{Q;\eta,k}(n)$ just depends on the residue of $n$ modulo $Q$, and, immediately from the definition, for any $r \in \{0,\dots,Q-1\}$, one has \begin{equation}\label{smoothingindeed} \sum_{\substack{n \in I_{\eta,k} \\ n \equiv r {\normalfont \hspace{0.8mm}  \text{mod} \hspace{0.8mm}} Q}} w^A_{Q;\eta,k}(n) = \sum_{\substack{n \in I_{\eta,k} \\ n \equiv r {\normalfont \hspace{0.8mm}  \text{mod} \hspace{0.8mm}} Q}} 1_A(n).\end{equation} The use of $w^A_{Q;\eta}$ comes from the fact that its Fourier transform models that of $A$ nearly perfectly on rationals with denominator dividing $Q$. As long as $Q$ is sufficiently composite (which we will choose it to be), we don't need to care much about other rationals, since the Fourier transform of the indicator function of the squares will be sufficiently small there. 
\end{remark}

\vspace{1.5mm}

For the following lemma, fix $Q,N \in \N, \eta \in \frac{1}{\N}$, and $A \sub [N]$. 

\vspace{1.5mm}

\begin{definition}
Define the \textit{balanced function} $f^A_{Q;\eta} : \N \to \R$ by $f^A_{Q;\eta} := 1_A-w^A_{Q;\eta}$. 
\end{definition}

\vspace{0.1mm}

\begin{lemma}\label{fourierapproximation}
Take some $a,q \in \N$ with $q \mid Q$. Then, for any $\beta \in \R$, it holds that $$\left|\wh{f^A_{Q;\eta}}\left(\frac{a}{q}+\beta\right)\right| \le 2|\beta|\eta N^2.$$
\end{lemma}

\begin{proof}
For $k \in \{0,\dots,\eta^{-1}-1\}$, define $f^A_{Q;\eta,k} := f^A_{Q;\eta}1_{I_{\eta,k}} = 1_{I_{\eta,k}}1_A-w^A_{Q;\eta,k}$ so that \begin{equation}\label{sumdecomp} f^A_{Q;\eta} = \sum_{k=0}^{\eta^{-1}-1} f^A_{Q;\eta,k}.\end{equation} Fix $a,q \in \N$ with $q \mid Q$, and fix $\beta \in \R$. By \eqref{sumdecomp}, linearity of the fourier transform, and the triangle inequality, to prove Lemma \ref{fourierapproximation} it suffices to show $$\left|\wh{f^A_{Q;\eta,k}}(\frac{a}{q}+\beta)\right| \le 2|\beta| \eta N |I_{\eta,k}| \vspace{1.5mm}$$ for each $k \in \{0,\dots,\eta^{-1}-1\}$. So fix some such $k$. By definition, \begin{equation}\label{explicitfAhat} \wh{f^A_{Q;\eta,k}}(\frac{a}{q}+\beta) = \sum_{n \in I_{\eta,k}} 1_A(n) e\left((\frac{a}{q}+\beta)n\right)-\sum_{n \in I_{\eta,k}} w^A_{Q;k}(n)e\left((\frac{a}{q}+\beta)n\right).\end{equation} Letting $L = \lfloor k\eta N \rfloor+1$ denote the left endpoint of $I_{\eta,k}$, we trivially from \eqref{explicitfAhat} have $$\left|\wh{f^A_{Q;\eta,k}}(\frac{a}{q}+\beta)\right| = \left|\sum_{n \in I_{\eta,k}} 1_A(n) e\left((\frac{a}{q}+\beta)(n-L)\right)-\sum_{n \in I_{\eta,k}} w^A_{Q;\eta,k}(n)e\left((\frac{a}{q}+\beta)(n-L)\right)\right|.$$ The reason for shifting the phase by $L$ is that if we now use $$\sum_{n \in I_{\eta,k}} 1_A(n) e\left(\frac{a(n-L)}{q}\right) - \sum_{n \in I_{\eta,k}} w^A_{Q;\eta,k}(n)e\left(\frac{a(n-L)}{q}\right) = 0$$ (which follows from \eqref{smoothingindeed} and that $q \mid Q$) to write $$\hspace{-20mm} \left|\wh{f^A_{Q;\eta,k}}(\frac{a}{q}+\beta)\right| = \Bigg|\sum_{n \in I_{\eta,k}} 1_A(n) \left[e\left((\frac{a}{q}+\beta)(n-L)\right)-e\left(\frac{a(n-L)}{q}\right)\right]$$ $$\hspace{40mm} -\sum_{n \in I_{\eta,k}} w^A_{Q;\eta,k}(n)\left[e\left((\frac{a}{q}+\beta)(n-L)\right)-e\left(\frac{a(n-L)}{q}\right)\right]\Bigg|,$$ then the trivial $|e(x)-e(y)| \le |x-y|$ is strong enough to give the sufficient bound \begin{align*}\left|\wh{f^A_{Q;\eta,k}}(\frac{a}{q}+\beta)\right| &\le \sum_{n \in I_{\eta,k}} 1_A(n)|\beta|(n-L)+\sum_{n \in I_{\eta,k}} |w^A_{Q;\eta,k}(n)|\hspace{0.5mm}|\beta|\hspace{0.5mm}(n-L) \\ &\le 2|\beta|\eta N|I_{\eta,k}|, \end{align*} the last inequality using that $n-L \le \eta N$ for each $n \in I_{\eta,k}$.
\end{proof}

\vspace{1mm}

\begin{remark}\label{theplan}
The plan to prove Theorem \ref{3/8+epsilon} is to decompose $$1_A*1_B = w^A_{Q;\eta}*w^B_{Q;\eta}+f^A_{Q;\eta}*w^B_{Q;\eta}+w^A_{Q;\eta}*f^B_{Q;\eta}+f^A_{Q;\eta}*f^B_{Q;\eta}$$ and use Lemma \ref{fourierapproximation} to argue that the ``number" of squares ``in" $1_A*1_B$ is approximately the same as that in $w^A_{Q;\eta}*w^B_{Q;\eta}$. The latter, involving the convolution of two functions constant on residues modulo $Q$, is more easily calculable and comes down to the weighted number of mod $Q$ quadratic residues in the convolution of the natural mod $Q$ projections of $w^A_{Q;\eta},w^B_{Q;\eta}$. The following (with Lemma \ref{fourierapproximation}) will be used to prove the validity of the approximation. 
\end{remark}

\vspace{1mm}

\begin{proposition}\label{squarecountfourieruniform}
Let $f,g : [N] \to [-1,1]$ be ($1$-bounded) functions. Suppose $\delta > 0$ is such that $\left|\wh{f}(\frac{a}{q}+\beta)\right| \le \delta |\beta| N^2$ for each $a,q \le \lam^{-2}$ and\footnote{We will only need the condition for $|\beta| \le \frac{\lam^{-2}}{2N}$.} $\beta \in \R$. Then we have $$\left|\sum_{n \ge 1} (f*g)(n)1_S(n)\right| \le 10(\delta\lam^{-8}+\lam)N^{3/2}.$$
\end{proposition}

\begin{proof}
We may replace $S$ by $S_{2N} := \{m^2 : m \in \N, m^2 \le 2N\}$ and write \begin{equation}\label{fourierintegers} \sum_{n \ge 1} (f*g)(n)1_{S_{2N}}(n) = \int_\T \wh{f}(\theta)\wh{g}(\theta)\wh{1_{S_{2N}}}(-\theta)d\theta.\end{equation} We import the needed ``minor arc" estimate from \cite{lyall}: \begin{lemma}[\cite{lyall}, Proposition 1]\label{minorarcestimate} For any $\lam > 0$, if $N \in \N$ is sufficiently large and $\theta \in \T$ is such that $|\theta-\frac{a}{q}| > \frac{\lam^{-2}}{N}$ for each $a,q \le \lam^{-2}$, then $|\wh{1_{S_N}}(\theta)| \le 5\lam N^{1/2}$.\end{lemma} \noindent This lemma together with Cauchy-Schwarz and Plancherel immediately gives \begin{align*} \left|\int_{\mf{m}} \wh{f}(\theta)\wh{g}(\theta)\wh{1_{S_{2N}}}(-\theta)d\theta\right| &\le 5\lam \sqrt{2N} \int_{\mf{m}} |\wh{f}(\theta)||\wh{g}(\theta)|d\theta \\ &\le 10\lam N^{1/2}\left(\int_\T |\wh{f}(\theta)|^2d\theta\right)^{1/2}\left(\int_\T |\wh{g}(\theta)|^2d\theta\right)^{1/2} \\ &= 10\lam N^{1/2}\left(\sum_{n \le N} f(n)^2\right)^{1/2}\left(\sum_{n \le N} g(n)^2\right)^{1/2} \\ &\le 10 \lam N^{3/2},\end{align*} where $\mf{m}$ is defined so that $$\T\setminus \mf{m} := \bigcup_{q=1}^{\lam^{-2}}\bigcup_{\substack{1 \le a \le q \\ (a,q) = 1}} \left\{\theta \in \T : \left|\theta-\frac{a}{q}\right| \le \frac{\lam^{-2}}{2N}\right\}.$$ Letting $\beta_* = \frac{\lam^{-2}}{2N}$ for notational ease, we handle the ``major arc" as follows: \begin{align*}\left|\int_{\T\setminus \mf{m}} \wh{f}(\theta)\wh{g}(\theta)\wh{1_{S_{2N}}}(-\theta)d\theta \right| &\le \sum_{q=1}^{\lam^{-2}}\sum_{1 \le a \le q} \left|\int_{\frac{a}{q}-\beta_*}^{\frac{a}{q}+\beta_*} \wh{f}(\theta)\wh{g}(\theta)\wh{1_{S_{2N}}}(-\theta)d\theta\right| \\ &\le \sum_{q=1}^{\lam^{-2}}\sum_{1 \le a \le q}\int_{-\beta_*}^{\beta_*} \delta |\beta| N^2 N \sqrt{2N}d\beta \\ &\le \sqrt{2}\delta N^{7/2}\sum_{q=1}^{\lam^{-2}} \sum_{1 \le a \le q} 2\beta_*^2 \\ &\le 10\delta \lam^{-8} N^{3/2}.\end{align*} (The bound ``10" here is loose and used for simplicity.) We're done by \eqref{fourierintegers}. 
\end{proof}

\vspace{1.5mm}

To complete the plan outlined in Remark \ref{theplan}, we need to argue that $w^A_{Q;\eta}*w^B_{Q;\eta}$ ``contains" many squares. We start by focusing on particular intervals. We abstract out from our exact the situation the relevant property of $w^A_{Q;\eta,k}$ and $w^B_{Q;\eta,k}$. 

\vspace{1.5mm}

\begin{proposition}\label{particularintervals}
Fix $\ep > 0$ and $Q \ge 1$. Let functions $\ol{w}_1,\ol{w}_2 : \Z/Q\Z \to [0,1]$ satisfy $$\sum_{t \in \Z/Q\Z} \ol{w}_i(t) \ge \left(\frac{3}{8}+\ep\right)Q$$ for $i=1,2$. For large $M \in \N$ and intervals $I_i = [k_iM,(k_i+1)M]$, $i=1,2$, define $$w_i(n) := 1_{I_i}(n) \hspace{0.5mm} \ol{w}_i(n {\normalfont \hspace{1.3mm}  \text{mod} \hspace{1.15mm}} Q)$$ for $i=1,2$. Then we have the lower bound $$\sum_{n \ge 1} (w_1*w_2)(n)1_S(n) \ge \frac{1}{200}c(\ep)\frac{M^{3/2}}{\sqrt{k_1+k_2}},$$ where $c(\ep)>0$ is the constant guaranteed by Theorem \ref{modular}.
\end{proposition}

\begin{proof}
Let $$J = \left[(k_1+k_2+1)M-\frac{1}{10}M,(k_1+k_2+1)M+\frac{1}{10}M\right]$$ so that for any $n \in J$ and $a \in \{0,\dots,Q-1\}$, it holds that $$\#\left\{m \in I_1 : m \equiv a \modd Q \text{ and } n-m \in I_2\right\} \ge \frac{1}{10}\frac{M}{Q}$$ (provided $M$ is large enough). Therefore, \begin{align*} \sum_{n \ge 1} (w_1*w_2)(n)1_S(n) &\ge \sum_{n \in J} \sum_{\substack{m \in I_1 \\ n-m \in I_2}} w_1(m)w_2(n-m)1_S(n) \\ &= \sum_{\substack{n \in J \\ n \in S}} \sum_{a=0}^{Q-1} \ol{w}_1(a)\ol{w}_2(n-a \modd Q)\sum_{\substack{m \equiv a \modd Q \\ m \in I_1 \\ n-m \in I_2}} 1 \\ &\ge \frac{M}{10Q}Q\sum_{\substack{n \in J \\ n \in S}} (\ol{w}_1*\ol{w}_2)(n \modd Q) \\ &= \frac{M}{10}\sum_{t=0}^{Q-1} (\ol{w}_1*\ol{w}_2)(t)\cdot \#\{m \in \N : m^2 \in J \text{, } m^2 \equiv t \modd Q\}. \end{align*} Note that, for $\ol{J} := \{m \in \N : m^2 \in J\}$, we have as $M \to \infty$ that $$\#\{m \in \N : m^2 \in J \text{, } m^2 \equiv t \modd Q\} = \left(1+o(1)\right)f_Q(t)\frac{\left|\ol{J}\right|}{Q}.$$ We lower-bound \begin{align*} \left|\ol{J}\right| &\ge \frac{1}{2}\left(\sqrt{(k_1+k_2+1)M+\frac{1}{10}M}-\sqrt{(k_1+k_2+1)M-\frac{1}{10}M}\right) \\ &= \frac{1}{2}\frac{\frac{2}{10}M}{\sqrt{(k_1+k_2+1)M+\frac{1}{10}M}+\sqrt{(k_1+k_2+1)M-\frac{1}{10}M}} \\ &\ge \frac{1}{2} \frac{\frac{1}{10}M}{\sqrt{(k_1+k_2)M}}.\end{align*} Combining everything, we obtain $$\sum_{n \ge 1} (w_1*w_2)(n)1_S(n) \ge \frac{M}{10Q}\frac{\sqrt{M}}{20\sqrt{k_1+k_2}}\sum_{t=0}^{Q-1} (\ol{w}_1*\ol{w}_2)(t)f_Q(t).$$ By the assumptions of the current theorem, Theorem \ref{modular} finishes the proof. 
\end{proof}

\vspace{1.5mm}

Back to our specific setting, we can now handle $w^A_{Q;\eta}*w^B_{Q;\eta}$. 

\vspace{1.5mm}

\begin{proposition}\label{waconvolutionwb}
Fix $\ep>0, Q \in \N$, and $\eta \in \frac{1}{\N}$. Then for all large $N \in \N$ and any $A,B \sub [N]$ with $|A|,|B| \ge (\frac{3}{8}+\ep)N$, we have $$\sum_{n \ge 1} (w^A_{Q;\eta}*w^B_{Q;\eta})(n)1_S(n) \ge \frac{\ep^2}{5000}c\left(\frac{\ep}{2}\right)N^{3/2},$$ where $c(\ep) > 0$ is the constant guaranteed by Theorem \ref{modular}. 
\end{proposition}

\begin{proof}
It is easy to see that $|A| \ge (\frac{3}{8}+\ep)N$ implies there are at least $\frac{\ep}{3}\eta^{-1}$ values of $k \in \{0,\dots,\eta^{-1}-1\}$ with $|A\cap I_{\eta,k}| \ge (\frac{3}{8}+\frac{3\ep}{4})|I_{\eta,k}|$. Therefore, by taking $N$ large enough, if we let\footnote{The choice of summing $n$ over $[\lfloor k \eta N \rfloor+1, \lfloor k \eta N\rfloor+Q]$ is arbitrary; any $Q$ numbers in $I_{\eta,k}$, all distinct modulo $Q$, would of course be equivalent.} $$J^A := \left\{k \in \{0,\dots,\eta^{-1}-1\} : \sum_{n=\lfloor k \eta N \rfloor+1}^{\lfloor k \eta N\rfloor +Q} w^A_{Q;\eta,k}(n) \ge \left(\frac{3}{8}+\frac{\ep}{2}\right)Q\right\},$$ then we have $|J^A| \ge \frac{\ep}{4}\eta^{-1}$. Defining $J^B$ in the analogous way, we by symmetry have $|J^B| \ge \frac{\ep}{4}\eta^{-1}$. The point is that Proposition \ref{particularintervals} (with $M = \eta N$) then lets us bound \begin{align*} \sum_{n \ge 1} (w^A_{Q;\eta}*w^B_{Q;\eta})(n)1_S(n) &= \sum_{k_1,k_2 = 0}^{\eta^{-1}-1} \sum_{n \ge 1} (w^A_{Q;\eta,k_1}*w^B_{Q;\eta,k_2})(n)1_S(n) \\ &\ge \sum_{\substack{k_1 \in J^A \\ k_2 \in J^B}} \sum_{n \ge 1} (w^A_{Q;\eta,k_1}*w^B_{Q;\eta,k_2})(n)1_S(n) \\ &\ge \sum_{\substack{k_1 \in J^A \\ k_2 \in J^B}} \frac{1}{200}c\left(\frac{\ep}{2}\right)\frac{(\eta N)^{3/2}}{\sqrt{k_1+k_2}} \\ &\ge \frac{1}{200}c\left(\frac{\ep}{2}\right)\frac{(\eta N)^{3/2}}{\sqrt{2\eta^{-1}}}|J^A|\hspace{0.5mm} |J^B|.\end{align*} The proof is complete by inserting the lower bounds $|J^A|,|J^B| \ge \frac{\ep}{4}\eta^{-1}$.
\end{proof}

\vs

We now put everything together to obtain (a more quantitative version of) our main theorem.

\vspace{1mm}

\setcounter{theorem}{0}
\begin{theorem}\label{3/8+epsilon}
For any $\ep > 0$, if $N$ is sufficiently large and $A,B \sub [N]$ have $|A|,|B| \ge (\frac{3}{8}+\ep)N$, then $A+B$ contains a perfect square. In fact, we have the quantitative $$\#\{(a,b) \in A\times B : a+b \in S\} \ge 10^{-6}\ep^3 N^{3/2}.$$
\end{theorem}

\begin{proof}
Let $\eta \in \frac{1}{\N},\ol{Q} \in \N$ be parameters (based on $\ep$) to be determined, and set $Q := \lcm(1,\dots,\ol{Q})$. Take $N$ sufficiently large and $A,B \sub [N]$ with $|A|,|B| \ge (\frac{3}{8}+\ep)N$. As remarked earlier, we decompose $$1_A*1_B = w^A_{Q;\eta}*w^B_{Q;\eta}+f^A_{Q;\eta}*w^B_{Q;\eta}+w^A_{Q;\eta}*f^B_{Q;\eta}+f^A_{Q;\eta}*f^B_{Q;\eta}.$$ Proposition \ref{waconvolutionwb} gives $$\sum_{n \ge 1} (w^A_{Q;\eta}*w^B_{Q;\eta})(n)1_S(n) \ge \frac{\ep^2}{5000}c\left(\frac{\ep}{2}\right)N^{3/2},$$ and Proposition \ref{squarecountfourieruniform} together with Lemma \ref{fourierapproximation} gives $$\left|\sum_{n \ge 1} (f^A_{Q;\eta}*w^B_{Q;\eta})(n)1_S(n)\right| \le 10\left(2\eta \ol{Q}^4 + \ol{Q}^{-1/2}\right)N^{3/2},$$ and the same bound for the analogous inequalities involving $w^A_{Q;\eta}*f^B_{Q;\eta}$ and $f^A_{Q;\eta}*f^B_{Q;\eta}$. Therefore, $$\sum_{n \ge 1} (1_A*1_B)(n)1_S(n) \ge \frac{\ep^2}{5000}c\left(\frac{\ep}{2}\right) N^{3/2} - 30\left(2\eta \ol{Q}^4 + \ol{Q}^{-1/2}\right) N^{3/2}.$$ Setting $\eta = \ol{Q}^{-9/2}$ and using $c(\ep) \ge \ep/3$, we obtain $$\sum_{n \ge 1} (1_A*1_B)(n)1_S(n) \ge \left(\frac{\ep^3}{30000}-90\ol{Q}^{-1/2}\right)N^{3/2}.$$ Choosing $\ol{Q}$ a perfect square (merely so that $\eta \in \frac{1}{\N}$) with $\ol{Q}^{-1/2} \le 10^{-7}\ep^3$, say, finishes the proof. 
\end{proof}

\vspace{1.5mm}

\section{Solving the Optimization Problem}\label{optimizationsection}

We finish the paper by proving the inequality that Theorem \ref{modular} relied upon. It could be verified directly by a computer but would take quite a bit of time. 

\vs

For $a_0,\dots,a_{23} \in [0,1]$, we let $a : \Z/24\Z \to [0,1]$ be given by $a(i) = a_i$. Recall, for $a,b \in \Z/24\Z$ and $t \in \Z/24\Z$, we define $$(a*b)(t) := \frac{1}{24}\sum_{i \in \Z/24\Z} a(i)b(t-i)$$ $$f_{24}(t) := \#\{j \in \Z/24\Z : j^2 \equiv t \modd 24\}.$$

\vs

In this section, we prove the following, stated previously in Section \ref{modularsection}. 

\begin{manualtheorem}{3.6}
For any $\ep > 0$, there is some $c'(\ep) > 0$ so that the following holds. For all $a_0,\dots,a_{23},b_0,\dots,b_{23} \in [0,1]$ with $\sum_{i=0}^{23} a_i \ge 9+\ep, \sum_{i=0}^{23} b_i \ge 9+\ep$, we have $$\sum_{t \in \Z/24\Z} (a*b)(t)f_{24}(t) \ge c'(\ep)+\frac{1}{\sqrt{5}}\sqrt{\sum_i a_i-\sum_i a_i^2}\sqrt{\sum_i b_i-\sum_i b_i^2}.$$ In fact, one can take $c'(\ep) = \frac{1}{\sqrt{5}}\ep$.  
\end{manualtheorem}

\vspace{1mm}

The proof, with $c'(\ep) = \frac{1}{\sqrt{5}}\ep$, will follow from the proof of the ``$\ep=0$" case, in which we also identify the extremizers. We say $a$ is a \textit{lift-up} of a subset $A$ of $\Z/8\Z$ if: $a_i = 1$ if and only if $i \modd 8 \in A$, and $a_i = 0$ otherwise.

\vspace{1.5mm}

\begin{proposition}\label{optimization}
For all $a_0,\dots,a_{23},b_0,\dots,b_{23} \in [0,1]$ with $\sum_{i=0}^{23} a_i \ge 9, \sum_{i=0}^{23} b_i \ge 9$, we have $$\sum_{t \in \Z/24\Z} (a*b)(t)f_{24}(t) \ge \frac{1}{\sqrt{5}}\sqrt{\sum_i a_i-\sum_i a_i^2}\sqrt{\sum_i b_i-\sum_i b_i^2}$$ with equality if and only if there is some $x \in \Z/8\Z$ so that $a,b$ are lift-ups of $\{0,1,5\}+x, \{2,5,6\}-x \sub \Z/8\Z$.
\end{proposition}

We prove Proposition \ref{optimization} by first massaging the desired inequality into a homogeneous quadratic form. It is of course easy to check the ``if" implication of the equality part of Proposition \ref{optimization}; the ``only if" direction will follow from equality needing to hold at each step of the proof and equality holding only for the claimed extremizers at the end of the proof. 

\vs

By the arithmetic-geometric inequality, it suffices to show $$\sum_{t \in \Z/24\Z} (a*b)(t)f_{24}(t) \ge \frac{1}{2\sqrt{5}}\left(\sum_i a_i - \sum_i a_i^2 + \sum_i b_i - \sum_i b_i^2\right)$$ for all $a_i,b_i \in [0,1]$ with $\sum_i a_i, \sum_i b_i \ge 9$. Since\footnote{If $x,y \ge 9+\ep$, then $\frac{2}{9}xy \ge x+y+2\ep$, which is why $c'(\ep) := \frac{1}{\sqrt{5}}\ep$ suffices.} $\frac{2}{9}xy \ge x+y$ if $x,y \ge 9$, it suffices to show $$\sum_{t \in \Z/24\Z} (a*b)(t)f_{24}(t) \ge \frac{1}{2\sqrt{5}}\left(\frac{2}{9}(\sum_i a_i)(\sum_i b_i)-\sum_i a_i^2-\sum_i b_i^2\right)$$ for all $a_i,b_i \in [0,1]$ with $\sum_i a_i, \sum_i b_i \ge 9$. Of course it then suffices to prove the inequality for any non-negative reals $a_i,b_i$.

\vs

\begin{proposition}\label{quadraticform}
For any $a_0,b_0,\dots,a_{23},b_{23} \in [0,\infty)$ one has $$\sum_{t \in \Z/24\Z} (a*b)(t)f_{24}(t) \ge \frac{1}{2\sqrt{5}}\left(\frac{2}{9}(\sum_i a_i)(\sum_i b_i) - \sum_i a_i^2 - \sum_i b_i^2\right).$$
\end{proposition}

\vs

We will present a proof of Proposition \ref{quadraticform} due to Fedor Nazarov. The (quite ingenious) proof significantly reduces the computational power needed.

\begin{proof} \text{}

\vs

\un{Step 1}: Reduction to a norm inequality in a single (non-negative) variable.

\vs

\noindent Using that $$\sum_{t \in \Z/24\Z} (a*b)(t)f_{24}(t) = \sum_{t \in \Z/24\Z} (\wt{a}*f_{24})(t)b(t),$$ where $\wt{a}(i) := a(-i)$ and $$\left(\sum_i a_i\right)\left(\sum_i b_i\right) = 24\sum_{t \in \Z/24\Z} (\wt{a}*\mathbbm{1})(t)b(t),$$ where $\mathbbm{1} : \Z/24\Z \to [0,1]$ is the constant function $\equiv 1$, we wish to prove $$\sum_{t \in \Z/24\Z} \left(\wt{a}*(\frac{16}{3}\mathbbm{1}-2\sqrt{5}f_{24})\right)\hspace{-1mm}(t)\hspace{1mm} b(t) \le \sum_{t \in \Z/24\Z} \left[a(t)^2+b(t)^2\right].$$ We may, of course, ignore the distinction between $a$ and $\wt{a}$, so we drop the \hs $\wt{\text{}}$ \hs from here on\footnote{However, the reader should keep in mind that we are ``mirroring" the extremizers.}. Since $2xy \le x^2+y^2$ for all $x,y \in \R$, it suffices to show $$\sum_{t \in \Z/24\Z} \left(a*(\frac{16}{3}\mathbbm{1}-2\sqrt{5}f_{24})\right)\hspace{-1mm}(t)\hspace{1mm} b(t) \le 2\left(\sum_{t \in \Z/24\Z} a(t)^2\right)^{1/2}\left(\sum_{t \in \Z/24\Z} b(t)^2\right)^{1/2},$$ which we write more compactly as $$\langle a*\varphi,b \rangle \le 2\|a\|_2\|b\|_2,$$ with $\varphi := \frac{16}{3}\mathbbm{1}-2\sqrt{5}f_{24}$. Since $b(t) \ge 0$ for each $t$, it suffices to prove $$\Big\langle (a*\varphi)_+,b\Big\rangle \le 2\|a\|_2\|b\|_2.$$ By Cauchy-Schwarz, it then suffices to prove $$\left \| (a*\varphi)_+ \right \|_2 \le 2\|a\|_2$$ for each $a : \Z/24\Z \to [0,\infty)$.

\vsss

\un{Step 2}: Showing the maximizer is an eigenvector of a related operator.

\vs

\noindent By compactness, let $a = \wh{a}$ be a maximizer of $\|(a*\varphi)_+\|_2$ subject to $\|a\|_2 = 1$ and $a \ge 0$ (pointwise). Let $\wh{\sigma} : \Z/24\Z \to \R$ satisfy $|\wh{\sigma}(t)| < \wh{a}(t)$ whenever $\wh{a}(t) > 0$ (think $\wh{\sigma} \to 0$). Then \begin{align*} \left\|\left((\wh{a}+\wh{\sigma})*\varphi\right)_+\right\|_2^2 - \left\|(\wh{a}*\varphi)_+\right\|_2^2 &= \sum_{t \in \Z/24\Z} \Bigg[\Big((\wh{a}*\varphi)(t)+(\wh{\sigma}*\varphi)(t)\Big)_+^2-\Big((\wh{a}*\varphi)(t)\Big)_+^2\Bigg] \\ &= 2\sum_{t \in \Z/24\Z} \Big((\wh{a}*\varphi)(t)\Big)_+(\wh{\sigma}*\varphi)(t)+O\left(\|\wh{\sigma}\|^2\right) \\ &= 2\Big\langle (\wh{a}*\varphi)_+ \hs , \hspace{0.5mm} \wh{\sigma}*\varphi\Big\rangle + O\left(\|\wh{\sigma}\|^2\right) \\ &= 2\Big\langle (\wh{a}*\varphi)_+*\wt{\varphi}\hs , \hspace{0.5mm}\wh{\sigma}\Big\rangle + O\left(\|\wh{\sigma}\|^2\right),\end{align*} where the second equality used the fact that $(x+y)_+^2-x_+^2 = 2yx_++O(y^2)$ for any reals $x,y$ with $|y| < |x|$, and in the last equality, we again use the notation $\wt{\varphi}(\cdot) := \varphi(-\cdot)$. Let $\wh{v} : \Z/24\Z \to \R$ be $\wh{v} := (\wh{a}*\varphi)_+*\wt{\varphi}$ so that \vspace{1mm} $$\left\|\left((\wh{a}+\wh{\sigma})*\varphi\right)_+\right\|_2^2 - \left\|(\wh{a}*\varphi)_+\right\|_2^2 = 2\langle \wh{v},\wh{\sigma}\rangle + O\left(\|\wh{\sigma}\|^2\right). \vspace{1mm}$$ We see that no $t \in \Z/24\Z$ can satisfy $\wh{a}(t) = 0$ and $\wh{v}(t) > 0$, for otherwise we could let $\wh{\sigma}(t) = +\al$ for some (very) small $\al > 0$, $\wh{\sigma}(t') = -\delta$ for some $t'$ with $\wh{a}(t') > 0$ and appropriate $\delta > 0$ (which will be $O(\al^2)$), and $\wh{\sigma}=0$ elsewhere, to have $$||\wh{a}+\wh{\sigma}||_2 = 1 \hspace{3mm} \text{ and } \hspace{3mm} \left\|\left((\wh{a}+\wh{\sigma})*\varphi\right)_+\right\| > \left\|\left(\wh{a}*\varphi\right)_+\right\|,$$ contradicting the maximality of $\wh{a}$. And similarly no $t \in \Z/24\Z$ can satisfy $\wh{a}(t) > 0$ and $\wh{v}(t) \le 0$. Therefore, $\wh{v}_+$ is positive exactly when $\wh{a}$ is, and each are $0$ otherwise. This implies $$\wh{v}_+ \equiv \lambda \wh{a}$$ for some $\lambda > 0$, for otherwise one could make $2\langle \wh{v},\wh{\sigma}\rangle+O(\|\wh{\sigma}\|^2)$ negative for suitable small $\wh{\sigma}$, contradicting the maximality of $\wh{a}$. To end this step, quickly note \begin{align}\label{equalslambda} \left\|(\wh{a}*\varphi)_+\right\|_2^2 &= \Big \langle (\wh{a}*\varphi)_+,(\wh{a}*\varphi)_+\Big\rangle \\ \nonumber &= \Big \langle (\wh{a}*\varphi)_+, \wh{a}*\varphi \Big \rangle \\ \nonumber &= \langle \wh{v}, \wh{a} \rangle \\ \nonumber &= \langle \wh{v}_+, \wh{a}\rangle \\ \nonumber &= \lambda.\end{align}

\vsss

\un{Step 3}: Choosing a convenient norm.

\vs

\noindent We are given $\wh{a} : \Z/24\Z \to [0,\infty)$ satisfying $$\left((\wh{a}*\varphi)_+*\wt{\varphi}\right)_+ \equiv \lam \wh{a}$$ and, by \eqref{equalslambda}, we wish to show $\lam \le 4$. It suffices to find a function (``norm") $N : [0,\infty)^{\Z/24\Z} \to [0,\infty)$ satisfying the multiplicativity condition \begin{equation}\label{mult}N(\gamma a) = \gamma N(a)\end{equation} for all $\gamma \in [0,\infty)$ and $a : \Z/24\Z \to [0,\infty)$, and the two (dual) norm bounds \begin{equation}\label{n1} N\left((a*\varphi)_+\right) \le 2N(a)\end{equation} \begin{equation}\label{n2} N\left((a*\wt{\varphi})_+\right) \le 2N(a)\end{equation} for all $a : \Z/24\Z \to [0,\infty)$. Indeed, with such a norm $N$, we have $$\lam N(\wh{a}) = N(\lam \wh{a}) = N\Big(\left((\wh{a}*\varphi)_+*\wt{\varphi}\right)_+\Big) \le 2N\Big((\wh{a}*\varphi)_+\Big) \le 4N(\wh{a}).$$ Motivated by the (conjectured) extremizers, we use the norm $$N(a) := \max\left(9\|a\|_\infty,\|a\|_1\right).$$

\vs

\un{Step 4}: Showing the desired norm bounds.

\vs

\noindent It is clear that $N$ satisfies condition \eqref{mult}. To prove \eqref{n1}, we may normalize to $N(a)=9$ so that it suffices to show $$\left\{\begin{aligned} \|a\|_\infty \le 1 \\ \|a\|_1 \le 9\end{aligned}\hs\hs\right\} \implies \left\{\begin{aligned}\|(a*\varphi)_+\|_\infty \le 2 \\ \|(a*\varphi)_+\|_1 \le 18\end{aligned}\hs\hs\right\},$$ where, to recall, $$\varphi = \frac{16}{3}\mathbbm{1}-2\sqrt{5}f_{24}.$$ So take $a : \Z/24\Z \to [0,\infty)$ with $\|a\|_\infty \le 1$ and $\|a\|_1 \le 9$. Then we easily have $$\|(a*\varphi)_+\|_\infty \le \max_{t \in \Z/24\Z} \frac{1}{24}\sum_{j \in \Z/24\Z} a(j)\varphi(t-j) \le \frac{1}{24}\cdot \frac{16}{3}\cdot 9 = 2.$$ As $a \mapsto \|(a*\varphi)_+\|_1$ is convex, it simply suffices to check that $\|(a*\varphi)_+\|_1 \le 18$ for all $a \in \{0,1\}^{24} \sub [0,1]^{\Z/24\Z}$. We may assume WLOG that $a_0 = 1$, so that there are only $\sum_{k=0}^8 {23 \choose k} < 10^6$ cases to check, which is easily handled by a computer.

\vs

We do everything analogous to establish \eqref{n2} as well.

\vs

\noindent Below is the python code, presented in two columns to save space. 

\vs

\begin{Verbatim}[fontsize=\tiny]
import math
import itertools

f = []
for t in range(0,24):
    sum1 = 0
    for j in range(0,24):
        if ((j*j)%24 == t):
            sum1 = sum1+1
    f.append(sum1)
phi = []
for t in range(0,24):
    phi.append(16/3-2*math.sqrt(5)*f[t])
phit = []
for t in range(0,24):
    phit.append(phi[23-t])

def h(a,psi):
    sum1 = 0
    for t in range(0,24):
        sum2 = 0
        for j in range(0,24):
            sum2=sum2+a[j]*psi[(t-j)%24]
        sum2 = sum2/24
        sum2 = max(sum2,0)
        sum1 = sum1+sum2
    return sum1
\end{Verbatim}

\vspace{-90mm}
\begin{Verbatim}[fontsize=\tiny,xleftmargin=65mm]
c = []
for j in range(1,24):
    c.append(j)
max1 = 0
max2 = 0
for k in range(0,9):
    for A in itertools.combinations(c,k):
        A = list(A)
        A.insert(0,0)
        a = []
        for j in range(0,24):
            if (j in A):
                a.append(1)
            else:
                a.append(0)
        v1 = h(a,phi)
        v2 = h(a,phit)
        max1 = max(max1,v1)
        max2 = max(max2,v2)
        if (v1 >= 17.99):
            print ("extremizer - "+str(a))
        if (v2 >= 17.99):
            print ("extremizer for dual - "+str(a))
print (max1)
print (max2)
\end{Verbatim}

\vs

The output of the python code is as follows.

\vs

\begin{Verbatim}[fontsize=\tiny]
extremizer - [1, 1, 0, 0, 1, 0, 0, 0, 1, 1, 0, 0, 1, 0, 0, 0, 1, 1, 0, 0, 1, 0, 0, 0]
extremizer for dual - [1, 1, 0, 0, 0, 1, 0, 0, 1, 1, 0, 0, 0, 1, 0, 0, 1, 1, 0, 0, 0, 1, 0, 0]
extremizer for dual - [1, 0, 0, 1, 1, 0, 0, 0, 1, 0, 0, 1, 1, 0, 0, 0, 1, 0, 0, 1, 1, 0, 0, 0]
extremizer - [1, 0, 0, 1, 0, 0, 0, 1, 1, 0, 0, 1, 0, 0, 0, 1, 1, 0, 0, 1, 0, 0, 0, 1]
extremizer - [1, 0, 0, 0, 1, 1, 0, 0, 1, 0, 0, 0, 1, 1, 0, 0, 1, 0, 0, 0, 1, 1, 0, 0]
extremizer for dual - [1, 0, 0, 0, 1, 0, 0, 1, 1, 0, 0, 0, 1, 0, 0, 1, 1, 0, 0, 0, 1, 0, 0, 1]
18.000000000000004
18.000000000000004
\end{Verbatim}

\vs

Since we printed all $a$ for which $\|(a*\varphi)_+\|_1,\|(a*\wt{\varphi})_+|_1 \ge 17.99$ and the ones printed have $\|(a*\varphi)_+\|_1, \|(a*\wt{\varphi})_+\|_1 = 18$, the $+4\cdot 10^{-15}$ (added to $18$) is merely a computer-induced rounding error. 

\vs

We finish by analyzing the extremizers. We obtained only $3$ of the $8$ conjectured extremizers; however, we assumed WLOG that $a_0 = 1$. Translating the outputted extremizers indeed recovers all $8$ conjectured extremizers for $a$. Since such $a$ have $\sum_i a_i - \sum_i a_i^2 = 0$, the only extremizing $b$, for a given $a$, must satisfy $\sum_t (a*b)(t)f_{24}(t) = 0$, i.e., $a+b$ ``contains" no squares. Since all extremizers $a$ are translates of one another, we may focus on a particular extremizer $a$. Then, as is easily checked, $b$ is uniquely determined merely by ``process of elimination".   
\end{proof}

\vspace{1.5mm}

\section{Acknowledgments} 

I would like to thank my advisor Ben Green for suggesting this problem to me and Fedor Nazarov for nearly solving the optimization problem by hand.

\vs

\end{document}